\documentclass[a4paper,twoside,12pt]{article}
\usepackage{amssymb,amsfonts,amsmath,amsthm,latexsym}
\usepackage[pdftex,bookmarks,colorlinks=false]{hyperref}
\usepackage[hmargin=1.2in,vmargin=1.2in]{geometry}
\usepackage[auth-sc-lg,affil-sl]{authblk}
\usepackage{graphicx}
\usepackage{array,multirow}
\usepackage[font=small]{caption}
\usepackage[mathscr]{euscript}
\usepackage{enumitem}

\newtheorem{theorem}{Theorem}[section]
\newtheorem{corollary}[theorem]{Corollary}
\newtheorem{definition}[theorem]{Definition}
\newtheorem{lemma}[theorem]{Lemma}
\newtheorem{problem}[theorem]{Problem}
\newtheorem{proposition}[theorem]{Proposition}
\newtheorem{remark}[theorem]{Remark}
\newtheorem{conjecture}[theorem]{Conjecture}
\newtheorem{illustration}{Illustration}

\numberwithin{equation}{section}
\def\ni{\noindent}
\def\N{\mathbb{N}}

\def\C{\mathbb{C}}
\def\K{\mathbb{K}}
\def\cO{\mathscr{O}_n}

\title{\textbf{\sc  A Study on  Ornated Graphs}}

\author{Johan Kok}
\affil{\small Tshwane Metropolitan Police Department\\ City of Tshwane, Republic of South Africa\\ E-mail: kokkiek2@tshwane.gov.za}

\author{Sudev Naduvath}
\affil{\small Department of Mathematics\\ Vidya Academy of Science \& Technology, Thrissur, India.\\ E-mail: sudevnk@gmail.com}

\author{Vivian Mukungunugwa}
\affil{\small Department of Mathematics and Applied Mathematics,\\ University of Zimbabwe, City of Harare, Republic of Zimbabwe.\\ E-mail: vivianm@maths.uz.ac.zw}

\date{}

\begin{document}
\maketitle

\begin{abstract}
In this paper, we introduce the notion of a finite non-simple directed graph, called an ornated graph and initiate a study on ornated graphs.  An ornated graph is a directed graph on $n$ vertices, denoted by $O_n(s_l)$, whose vertices are consecutively labeled clockwise on the circumference of a circle and constructed from an ordered string $s_l$ joining them in such a way that for an odd indexed entry $a_t$ of the string, a tail $v_i$ has clockwise heads $v_j$ if and only if $(i+a_t) \ge j$ and for an even indexed entry $a_s$ of the string a tail $v_i$ has anticlockwise heads $v_j$ if and only if $(i-a_s) \le j$. The collection of the ornated graphs having this property is called the family of ornated graphs. Some interesting results are also presented in this paper on certain types of ornated graphs.
\end{abstract}

\ni \textbf{Keywords:} Degree sequence, ornated graph, Kyle graph, symmetric directed graph. 

\vspace{0.2cm}

\ni \textbf{Mathematics Subject Classification:} 05C07, 05C12, 05C20.

\section{Introduction}

For general notations and concepts in graph theory, we refer to \cite{BM}, \cite{FH} and \cite{DBW} and for digraph theory, we further refer to \cite{CL1} and \cite{JG1}. All graphs (and digraphs) mentioned in this paper are simple, connected and finite graphs, unless mentioned otherwise.

The set of all positive integers and set of all non-negative integers are denoted in this paper by $\N$ and $\N_0$ respectively. We denote a directed graph by $D$ and its underlying graph by $\underline{D}$. The \textit{order} and \textit{size} of a graph (or a digraph) is the number of vertices and the number edges (or arcs) in that graph (or digraph). The \textit{degree} of a vertex, $d(v_i)$ in a directed graph $D$, we explicitly refer to the degree found in the  underlying graph  denoted $\underline{D}$, hence $d_D(v_i) = d_{\underline{D}}(v_i)  = d^+_D(v_i) + d^-_D(v_i)$.

\section{Ornated Graphs} 

\begin{definition}\label{Def-OS}{\rm
An \textit{ordered string}, denoted by $s_l$, is defined as an $l$-tuple $(a_j)_{j=1}^{l}$ where $a_j\in \N_0\ , 1 \le j \le l$ for $j,l\in \N$. For brevity, we write $s_l = (a_j); 1 \le j \le l$. }
\end{definition}

Analogous to the corresponding terms in set theory, if the element $a_t$ is an entry in the ordered string $s_l$, then we write $a_t \in s_l$. For an entry of the ordered string say $a_t$, $t$ is called the \textit{index} of $a_t$. It should also be noted that the values $n$, $a_j$ and $l$ are independent from each other.
 
\vspace{0.2cm}
 
In view of the above concepts, we introduce the notion of ornated graphs as follows.

\begin{definition}\label{Def-OG}{\rm
Let $n$ be a positive integer and $s_l=(a_j); 1\le j \le l$  be an ordered string of non-negative integers. An \textit{ornated graph} on $n$ vertices, associated with an ordered string $s_l$, is denoted by $O_n(s_l)$ and is defined as a directed graph (digraph) with vertex set $V=V(O_n(s_l)) = \{v_i: i \in \N, i \le n \}$ and the arc set $A=A(O_n(s_l)) \subseteq V\times V $ such that for $i, j \in \N,~ (v_i,v_ j) \in A$ if and only if $(i+a_t) \ge j, i<j$, for odd indices $t$ and $(i-a_s) \le j,~ i>j$, for even indices $s$.

The collection of all ornated graphs with respect to a given ordered string $s_l$ is called the \textit{family of ornated graphs}, denoted by $\cO$. }
\end{definition} 

The following is another important concept on the ornated graphs on the same number of vertices.

\begin{definition}{\rm
When we apply the ordered strings $s_{l_1}, s_{l_2}$ consecutively to the same $n$ vertices we write $O_n(s_{l_1} + s_{l_2}) = O_n(s_{l_1}) + O_n(s_{l_2})$.} 
\end{definition}

\ni Invoking the above notation, we have the following result.

\begin{lemma}\label{Lem-OG1}
We have the following properties for ornated graphs. 
\begin{enumerate}\itemsep0mm
\item For the string $s_3=(a_1, a_2, a_3)$, we have the following properties.
\begin{enumerate}\itemsep0mm
\item[(i)] \textit{Associative law}: $O_n(a_1, a_2, a_3) = O_n(a_1, 0, 0) + O_n(0, a_2, a_3) = O_n(0, a_2, 0) + O_n (a_1, 0, a_3) = O_n(0, 0, a_3) + O_n(a_1, a_2, 0) $. 

\item[(ii)] \textit{Summation law:} $O_n(a_1, a_2, a_3) = O_n(a_1, 0, 0) +O_n(0, a_2, 0) + O_n(0, 0, a_3)$.
\end{enumerate}
\item For the string $s_4=(a_1, a_2, a_3, a_4)$, we have the following properties.
\begin{enumerate}\itemsep0mm
\item[(i)] \textit{Partial commutative law:} We have $O_n(a_1, a_2, a_3, a_4) = O_n(a_3, a_2, a_1, a_4) = O_n(a_1, a_4, a_3, a_2) = O_n(a_3, a_4, a_1, a_2)$.

\item[(ii)] \textit{Redundancy law:} If $s_m = (b_j);~ 1 \le j \le m$ is an ordered string reduced from $s_l = (a_j):~ 1 \le j \le l$ by removing all the zero entries from $s_l$ then their underlying graphs are isomorphic. 
\end{enumerate}
\end{enumerate}
\end{lemma}
\begin{proof}
$(i)$ From Definition \ref{Def-OG},  it follows that the ornated graph $O_n(a_1, 0, 0)$ has arcs $(v_i,v_ j) \in A(O_n(a_1, a_2, a_3))$ if and only $(v_i,v_ j) \in\{(v_i, v_j): (i+a_1) \ge j, i<j\}$.  By same definition, the ornated graph $O_n(0, a_2, a_3)$ has arcs $(v_i,v_ j) \in A(O_n(a_1, a_2, a_3))$ if and only $(v_i,v_ j) \in\{(v_i, v_j):(i+a_3) \ge j, i<j\}\cup \{(v_i, v_j):(i-a_2) \le j, i>j\}$. Hence, $A(O_n(a_1, a_2, a_3)) = A(O_n(a_1, 0, 0))\cup A(O_n(0, a_2, a_3))$. 

\vspace{0.2cm}

Therefore, it follows that $V(O_n(a_1, a_2, a_3)) = V(O_n(a_1, 0, 0)) \cup V(O_n(0, a_2, a_3)) = V(O_n(a_1, 0, 0) + O_n(0, a_2, a_3))$ and for the arc set, $A(O_n(a_1, a_2, a_3)) = A(O_n(a_1, 0, 0)) \cup A(O_n(0, a_2, a_3)) = A(O_n(a_1, 0, 0) + O_n(0, a_2, a_3))$ and hence the result follows.

\ni The other results ($(ii), (iii)$ and $(iv))$ follow in a similar way.
\end{proof}

\ni  We have the following observations as the direct and immediate consequences of Definition \ref{Def-OG} and Lemma \ref{Lem-OG1}.

\begin{remark}\label{Rem-1}{\rm 
Provided that the number of vertices and the values of even indexed and odd indexed entries respectively, remain the same and irrespective the positions of the even or odd indexed, except that they remain even or odd indexed, the number of ordered strings $(a_j);~ 1 \le j \le l$ with $l = m+t$ entries (with $m$ the number of even indexed entries and $t$ the number of odd indexed entries respectively), that construct identical ornated graphs, is given by $m!\,t!$. }
\end{remark}

\begin{remark}\label{Rem-2}{\rm 
If the number of vertices and the values of even indexed and odd indexed entries respectively, remain the same, and the positions of the entries (even or or indexed) are disregarded then the number of ordered strings that construct isomorphic underlying ornated graphs is given by $l!$} 
\end{remark}

We now introduce another notion called the reduced degree-string with respect to a given ordered string as follows.   

\begin{definition}{\rm
Consider the ordered string $s_l = (a_j);~ 1 \le j \le l$. If entry $a_t = \min\{a_j: 1 \le j \le l\}$, then after removing all entries $a_t$ the ordered string $s_{l_1} = (b_j);~ 1 \le j \le l_1$,  is called the \textit{reduced degree-string} with respect to $(a_j);~ 1 \le j \le l$. 

Recursively, the reduced degree-string $s_{l_2} = (c_j);~ 1 \le j \le l_2$ with respect to $(b_j);~ 1 \le j \le l_1$ can be found. The smallest reduced degree-string is always the string $(a_q, a_q, a_q, \ldots, a_q), a_q = \max\{(a_j); 1 \le j \le l\}$. }
\end{definition}

Now, we introduce the notion of the Kyle graph of a given family of ornated graphs as given below.

\begin{definition}{\rm 
The \textit{Kyle graph} of a given family $\cO$ of ornated graphs on $n$ vertices is the smallest ornated graph in $\cO$, attaining maximal degree associated with a given ordered string $s_l$.}
\end{definition}    

The existence of a Kyle graph for a given family of $\cO$ of ornated graph $O_n(S_l)$ is established in the following result.

\begin{lemma}\label{Lem-KG2}
Let $s_l= (a_j);~ 1 \le j \le l$ be a given ordered string. Then, the Kyle graph which attain the maximum possible degree $\Delta= \Delta(O_n(s_l))=2\sum \limits_{i=1}^{l}a_j$ at a single vertex can be constructed on $n=2k+1$, vertices where $k=a_q=\max\{(a_j); 1 \le j \le l\}$.
\end{lemma}
\begin{proof}
Let $k = a_q = \max\{(a_j); 1 \le j \le l\}$. Then, $2k+1$ is always a unique odd number and hence $d(v_{\lceil\frac{2k+1}{2}\rceil})$ is also unique. The vertex $v_{k+1}$ can be considered the unique central vertex of the ornated graph $O_{2k+1}(s_l)$. 

Clearly, beginning with the arc-less graph on $2k+1$ vertices and without loss of generality, let $r$ be an odd index and then linking only the arcs $\{(v_i, v_j):(i+a_r) \ge j, i<j\}$  where $a_r\le k$, we have $\Delta(O_{2k+1}(0, 0, 0,\ldots, a_r, \ldots, 0)) = 2a_r$. Now, set the value of $t$ recursively as $t=a_m$,where $m=1,2,3,\ldots,(r-1), (r+1),\ldots,l$. Hence, by linking the arcs $\{(v_i, v_j):(i+t) \ge j, i<j\}$ (odd index) and $\{(v_i, v_j):(i-t) \le j, i>j\}$ (even index) recursively, it follows that $\Delta(O_{n=2k+1}(s_l)) = 2\sum \limits_{j=1}^{l}a_j$, uniquely at vertex $v_{k+1}$. This completes the proof.
\end{proof}

The Kyle graph of the family of ornated graphs corresponding to the ordered string $(1,3)$ is illustrated in Figure \ref{fig:figure}.

\begin{figure}[h!]
\centering
\includegraphics[width=0.7\linewidth]{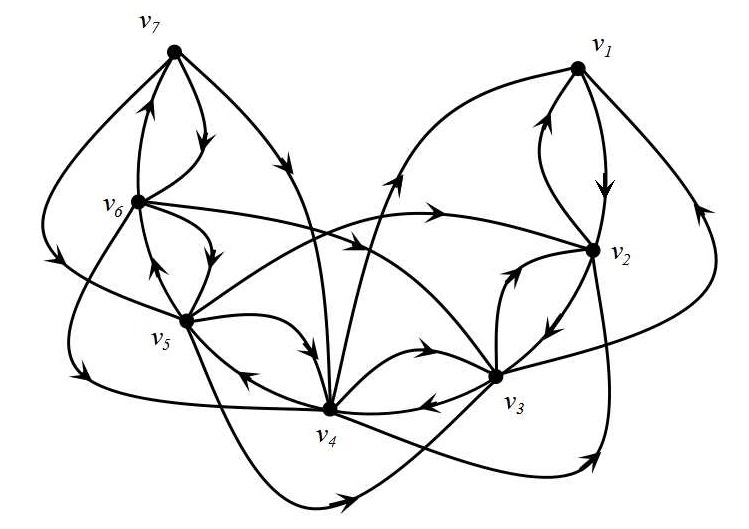}
\caption{}
\label{fig:figure}
\end{figure}

The following result determines the number of vertices with maximum degree in an ornated graph.

\begin{proposition}\label{Prop-OG3a}
For $n\in \N$, consider the family of ornated graphs with respect to the ordered string $s_l= (a_j);~ 1 \le j \le l$ and let $k=a_q=\max\{(a_j); 1 \le j \le l\}$. Then, the ornated graph $O_{(2k+1) + r}(s_l); r \in\N_0\ $, has exactly $(r+1)$ vertices with the maximum possible degree $\Delta(O_{(2k+1) + r}(s_l))=2\sum \limits_{j=1}^{l}a_j$.
\end{proposition}
\begin{proof}
Let $s_l= (a_j);~ 1 \le j \le l$ be the given ordered string and let $k=a_q=\max\{(a_j); 1 \le j \le l\}$. By Lemma \ref{Lem-KG2}, the vertex $v_{k+1}$ has degree $d(v_{k+1}) = \Delta(O_{(2k+1)}(s_l)) = 2\sum \limits_{j=1}^{l}a_j$.
 
Now, consider the ornated graph $O_{(2k+1)+1}(s_l)$. It can be noted that the vertex $v_{k+1}$ serves as the central vertex with respect to vertices $v_1, v_2, v_3, \ldots, v_k$ in applying Definition \ref{Def-OG} for all odd indexed vertices. It also serves as the central vertex for vertices $v_{k+2}, v_{k+3}, \ldots, v_{2k}$ in applying Definition \ref{Def-OG} to for all even indexed vertices. Hence,  $ d(v_{k+1}) = \Delta(O_{(2k+1) + 1}(s_l)) = 2\sum \limits_{i=1}^{l}a_i$. 

Similarly, $v_{k+2}$ is the central vertex with respect to vertices $v_{k+3}, v_{k+4}, \ldots, v_{2k+1}$ in applying Definition \ref{Def-OG} for all even indexed vertices and also serves as central vertex for vertices $v_2, v_3, \ldots, v_{k+1}$ in applying Definition \ref{Def-OG} for all odd indexed vertices. Hence, $d(v_{k+2}) = \Delta(O_{(2k+1) + 1}(s_l)) = 2\sum \limits_{j=1}^{l}a_j$. 

Hence, $d(v_{k+1}) = d(v_{k+2}) = \Delta(O_{(2k+1) + 1}(s_l)) = 2\sum \limits_{j=1}^{l}a_j$. The result follows recursively, for $O_{(2k+1) + r}(s_l), r \in\N_0\ $.
\end{proof}

\ni Invoking the above results, we define the following notion.

\begin{definition}{\rm
The set of vertices in the ornated graph $O_{(2k+1) + r}(s_l), r \in\N_0\ $, with degree  $\Delta(O_n(s_l))$ is called the \textit{central cluster} of the ornated graph and is denoted by $\C(O_n(s_l))$. }
\end{definition}

The properties of the vertices having the maximal possible degree in an ornated graph have already been discussed. An interesting question that arises in this context is about the vertices of these ornated graphs having minimum degree. As a result, we have the following.

\begin{lemma}\label{Lem-OG3}
For all ornated graphs from a family of finite ornated graphs $\cO$ which are larger or equal in order to the Kyle graph of that family, we have the vertices $v_1$ and $v_n$ only such that $d(v_1)=d(v_n)=\delta(O_n(s_l))=\sum \limits_{j=1}^{l}a_j$.
\end{lemma}
\begin{proof}
Let $s_l=(a_j);1 \le j \le l$ be an ordered string. The proof for the statement $d(v_1)=d(v_n)= \delta(O_n(s_l)) = \sum \limits_{j=1}^{l}a_j$, follows similarly to the proof of Lemma \ref{Lem-KG2}. 

Assume that there exists another vertex $v_i$ with $d(v_i) = \delta(O_n(s_l))$. Also assume, without loss of generality, that $1< i < m$, where $m$ is the smallest index with $v_m \in \C$. Beginning with the arcless graph and linking the arcs of an entry say, $a_t$ it follows that $d(v_2) = d(v_1) + 1$ in the subgraph of the ornated graph. It also follows that $d(v_2) \le d(v_3) \le d(v_4) \le \ldots \le d(v_i) \ldots \le d(v_m)$. The latter holds true as the linkages for all entries are added resulting in $d(v_1) < d(v_2) < d(v_3) < d(v_4) < \ldots < d(v_i) \ldots < d(v_m) = \Delta(O_n(s_l))$. This is a contradiction to the assumption that $d(v_i) = \delta(O_n(s_l))$.
\end{proof}

The following result is an immediate consequence of Proposition \ref{Prop-OG3a} and Lemma \ref{Lem-OG3}.

\begin{corollary}
For any ornated graph $O_n(s_l)$ in a family $\cO$ of finite ornated graphs with order greater than  or equal to that of a Kyle graph of that family,  $\Delta(O_n(s_l))=2\delta(O_n(s_l))$.
\end{corollary}
\begin{proof}
Since $\delta(O_n(s_l)) = \sum \limits_{j=1}^{l}a_j$ and $\Delta(O_{n=(2k+1) +r}(s_l)) = 2\sum \limits_{j=1}^{l}a_j$, for an ornated graph on at least $2k+1$ vertices,where $k = max\{(a_j); 1 \le j \le l\}$, we have $\Delta( O_n(s_l)) = 2\delta(O_n(s_l))$.
\end{proof}

The degree sequence in the Kyle graph of the family of ornated graphs $\cO$ is recursively determined in the following theorem. 

\begin{theorem}\label{Thm-OG4}
Let $s_l=(a_j); 1 \le j \le l$ be an ordered string. For the Kyle Graph, $O_{2k + 1}(s_l), k = a_q = max\{(a_j); 1 \le j \le l\}$ and $a_j \ge 1~ \forall j$ the degree sequence is given by 
\begin{eqnarray*}
d(v_i) =
\begin{cases}
\sum \limits_{j=1}^{l}a_j + (i-1)l, & 1\le i\le h_i+1; h_1= \min\{(a_j);1 \le j \le l\}\\
\sum \limits_{j=1}^{l}a_j + h_1l) + (i-1)l_1; & h_1+2\le i \le h_2-h_1; h_2= 2nd \min \{(a_j);1 \le j \le l\}\\
\vdots\\
(recursively)\\
\vdots\\
2\sum \limits_{j=1}^{l}a_j -((k - k_1) - 1); & i=k_1+2; k_1= 2nd \max \{(a_j);1 \le j \le l\}\\
2\sum \limits_{i=1}^{l}a_i -1; & i=k\\
2\sum\limits_{j=1}^{l}a_j; & i=k+1\\
\end{cases}
\end{eqnarray*}
and $v_{k+1}$ being the central vertex, the degrees of the remaining vertices are recursively given by $d(v_{(2k+1)-i}) = \sum \limits_{j=1}^{l}a_j + il; 0\le i\le h_1-1 ~\text{and}~ h_1=\min\{(a_j);1 \le j \le l\}$.
\end{theorem}
\begin{proof}
We prove the result for vertices $v_1, v_2, v_3, \ldots v_{k+1}$ since the mirror image of the degree sequence then follows from Definition \ref{Def-OG}. 

\vspace{0.2cm}

Consider the ordered string $s_l = (a_j); 1 \le j \le l$. Since the degree sequence of the Kyle Graph, $O_{2k + 1}(s_ l), k = a_q = \max\{(a_j); 1 \le j \le l\}$ is identical to the degree sequence of the underlying Kyle Graph we use Remark \ref{Rem-2} and consider the ordered string $(b_j);1 \le j \le l$ such that $b_j \ge b_\ell$ if and only if $j > l$ and $b_i = a_j$ for at least one value of $j$. Hence, $d(v_1) = \sum \limits_{i=1}^{l}b_i$ follows from Lemma \ref{Lem-OG3}. 

Consider the reduced string $s_1 = (b_1)$. Invoking Definition \ref{Def-OG}, we note that for the underlying Kyle Graph, $\underline O_{(2b_1+1)}(b_1)$ we have that $d(v_1) = b_1 = \sum \limits_{i=1}^{1}b_i$. Since $h_1 = b_1$ and $l=1$ it follows from Definition \ref{Def-OG} that $d(v_2) = b_1 + 1 = \sum \limits_{i=1}^{1}b_i + l , d(v_3) = d(v_2) + 1 = b_1 +2.1 = \sum \limits_{i=1}^{1}b_i +2l, \ldots, d(v_{b_1 + 1}) = d(v_{b_1}) + 1 = b_1 + b_1.1 = \sum \limits_{i=1}^{1}b_i +h_1l$. Hence, the result holds for $\underline O_{(2b_1+1)}(b_1)$.

\vspace{0.2cm}

\ni Now we consider the following cases.

\vspace{0.2cm}

\ni {\em Case 1:} Consider the reduced string $s_2 = (b_1, b_2), b_1 = b_2$. It follows that for the underlying Kyle Graph, $\underline O_{(2b_1+1)}(b_1, b_2)$ we have that $d(v_1) = 2b_1 = 2(\sum \limits_{i=1}^{1}b_i) = \sum \limits_{i=1}^{2}b_i$. It also follows that since $h_1 = b_2$ and $l=2$ for this case, we have $d(v_2) = 2(b_1 + 1) = 2(\sum \limits_{i=1}^{1}b_i + 1) = 2(b_2 +1) = \sum \limits_{i=1}^{2}b_i + 2 = \sum \limits_{i=1}^{2}b_i + l, d(v_3) = 2(d(v_2) + 1) = 2(b_1 +2) = 2(\sum \limits_{i=1}^{1}b_i +2) = 2(b_2 + 2) = \sum \limits_{i=1}^{2}b_i + 4 = \sum \limits_{i=1}^{2}b_i +2l, \ldots, d(v_{b_1 + 1}) = 2(d(v_{b_1}) + 1) = 2(b_1 + b_1) = 2(\sum \limits_{i=1}^{1}b_i +b_1) = 2(b_2 + b_2) = 2\sum \limits_{i=1}^{1}b_i + 2b_2 = \sum \limits_{i=1}^{2}b_i + h_1l$. So the result holds for $\underline O_{(2b_1+1)}(b_1, b_2), b_1 = b_2$.

\vspace{0.2cm}

\ni {\em Case 2:} Consider the reduced string $s_2 = (b_1, b_2), b_1 > b_2$. Let $b_1-b_2=t$. Definition \ref{Def-OG} allows the underlying graph of $O_{2b_1 +1}(0, b_2)$ to have degree sequence $d(v_1) = b_2$, $d(v_2) = b_2 +1$, $d(v_3) = b_2 +2$, and proceeding like this, $d(v_{(2b_1 +1) - t}) = b_2 + b_2 =2b_2$. Thereafter, for $j=1,2,3, \ldots,t$, we have $d(v_{(2b_1 +1) - (t-j)}) = 2b_2$. Hence, it follows that the underlying graph of $O_{2b_1 +1}(b_1, b_2)$ has degree sequence, $d(v_1)=\sum \limits_{i=1}^{2}b_i, d(v_2) = \sum \limits_{i=1}^{2}b_i + 2, d(v_3) = \sum \limits_{i=1}^{2}b_i + 4,\ldots\ldots, d(v_{(2b_1 +1) - t}) = \sum \limits_{i=1}^{2}b_i + 2b_2$. Thereafter for $j = 1, 2, 3, \ldots, t$, we have $d(v_{(2b_1 +1) - (t-j)})=\sum \limits_{i=1}^{2}b_i + j$. Hence, $h_1 = b_2$ and $l=2$ for vertices $i = 1,2,3, \ldots, (2b_1-t)$ and for the subsequent $t$ vertices, $l_1=1$ and for first $h_1+1$ vertices, we have $d(v_1) = \sum\limits_{i=1}^{l}a_i$, $d(v_2)=\sum \limits_{i=1}^{l}a_i + l$, $d(v_3)=\sum \limits_{i=1}^{l}a_i + 2l$ and so on up to $d(v_{h_1 +1})=\sum \limits_{i=1}^{l}a_i + h_1l$ where $h_1= \min\{(a_j); 1 \le j \le l\}$, for the next $h_2-h_1$ vertices, we have $d(v_{(h_1+2)})=\sum \limits_{i=1}^{l}a_i + h_1l + l_1$,  $d(v_{(h_1+3)})=\sum \limits_{i=1}^{l}a_i + h_1l+2l_1$, and so on up to $d(v_{(h_2- h_1)})= \sum \limits_{i=1}^{l}a_i + h_1l + (h_2 -h_1)l_1$, where $h_2= 2^{nd} \min\{(a_j);1 \le j \le l\}$, and recursively proceeding like this, for sequential reduced degree-strings until we have $d(v_{k_1+2})= 2\sum \limits_{i=1}^{l}a_i-((k - k_1) - 1), \ldots\ldots, d(v_k)=  2\sum \limits_{i=1}^{l}a_i-1, d(v_{k+1})= 2\sum \limits_{i=1}^{l}a_i$ for $(k - k_1)$ vertices, where $k_1 = 2nd \max\{(a_j); 1 \le j \le l\}$, with $v_{k+1}$ being central vertex. Then, recursively \textit{mirror image} degree values follow.  

Hence, the result follows recursively for the underlying graphs of $O_{2b_1 +1}(b_1, b_2, b_3)$,  $O_{2b_1 +1}(b_1, b_2, b_3, b_4)$, so on up to $O_{2b_1+1}((b_j);1 \le j \le l)$.
\end{proof}

\begin{illustration}{\rm 
For the Kyle Graph $O_{17}(1, 3, 5, 1, 2, 8)$ the degree sequence is given by $d(v_1) = 20, d(v_2) = 26, d(v_3) = 30, d(v_4) = 33, d(v_5) = 35, d(v_6) = 37, d(v_7)= 38, d(v_8) = 39, d(v_9) = 40, d(v_{10}) = 39, d(v_{11}) = 38, d(v_{12}) = 37, d(v_{13}) = 35, d(v_{14}) = 33, d(v_{15}) = 30, d(v_{16}) = 26$, and $d(v_{17}) = 20$. Let the notation $d(v_i) \rightsquigarrow a_j$ for illustrative purposes, denote $d(v_i)$ in $O_n(0,0,0, ..a_j, ..,0)$. The degree sequence can easily be calculated by using the summation law (Lemma \ref{Lem-OG1} (ii)). See the table below:

\begin{table}[h]
\begin{center}
\begin{tabular}{|m{0.5cm}|m{1.25cm}|m{1.25cm}|m{1.25cm}|m{1.25cm}|m{1.25cm}|m{1.25cm}|m{1cm}|}
\hline
$v_i$ & $d(v_i)\rightsquigarrow a_1$ & $d(v_i) \rightsquigarrow a_2$ & $d(v_i) \rightsquigarrow a_3$ & $d(v_i) \rightsquigarrow a_4$ & $d(v_i) \rightsquigarrow a_5$ & $d(v_i) \rightsquigarrow a_6$ & $d(v_i)$\\
\hline
$v_1$ & 1 & 3 & 5 & 1 & 2 & 8 & 20\\
\hline
$v_2$ & 2 & 4 & 6 & 2 & 3 & 9 & 26\\
\hline
$v_3$ & 2 & 5 & 7 & 2 & 4 & 10 & 30\\
\hline
$v_4$ & 2 & 6 & 8 & 2 & 4 & 11 & 33\\
\hline
$v_5$ & 2 & 6 & 9 & 2 & 4 & 12 & 35\\
\hline
$v_6$ & 2 & 6 & 10 & 2 & 4 & 13 & 37\\
\hline
$v_7$ & 2 & 6 & 10 & 2 & 4 & 14 & 38\\
\hline
$v_8$ & 2 & 6 & 10 & 2 & 4 & 15 & 39\\
\hline
$v_9$ & 2 & 6 & 10 & 2 & 4 & 16 & 40\\
\hline
$v_{10}$ & 2 & 6 & 10 & 2 & 4 & 15 & 39\\
\hline
$v_{11}$ & 2 & 6 & 10 & 2 & 4 & 14 & 38\\
\hline
$v_{12}$ & 2 & 6 & 10 & 2 & 4 & 13 & 37\\
\hline
$v_{13}$ & 2 & 6 & 9 & 2 & 4 & 12 & 35\\
\hline
$v_{14}$ & 2 & 6 & 8 & 2 & 4 & 11 & 33\\
\hline
$v_{15}$ & 2 & 5 & 7 & 2 & 4 & 10 & 30\\
\hline
$v_{16}$ & 2 & 4 & 6 & 2 & 3 & 9 & 26\\
\hline
$v_{17}$ & 1 & 3 & 5 & 1 & 2 & 8 & 20\\
\hline
\end{tabular}
\end{center}
\end{table}
}
\end{illustration}

Note the ``mirror image" degree sequence with $v_9$ being central to the degree sequence.

\begin{remark}{\rm 
Since the number of edges in a graph $G$ half of the sum of degrees of all vertices of $G$, we can find the number of edges in the Kyle graph of a family of ornated graphs $\cO$  by adding all the terms explained in Theorem \ref{Thm-OG4}. }
\end{remark}

\section{Determining the Ornated Graphs}

\begin{definition}{\rm
Consider vertices $v_i$ and $v_{i+1}$ of an ornated graph.  The \textit{relative indegree and outdegree} of a vertex $v_i$, denoted by $d_r^-(v_i)$ and $d_r^+(v_i)$ respectively, are defined to be the number of $(v_{i+1}, v_i)$ and $(v_i, v_{i+1})$ arcs, respectively.}
\end{definition}

The following result is on the number odd and even indexed entries in an ordered string related to the indegree and outdegree of a vertex in the corresponding ornated graph. 

\begin{proposition}\label{P-OGIOD}
For a vertex $v_i, i< n$ in an ornated graph, $O_n((a_j); 1 \le j \le l)$ which is larger or equal to the Kyle Graph, if $d_r^- (v_i)= m$ and $d_r^+(v_i) = t$, then the ordered string has $m+t$ entries composed of $m$ even indexed entries and $t$ odd indexed entries.
\end{proposition}
\begin{proof}
Construct the ornated graph $O_n((a_j); 1 \le j \le l), l= m+t$ with $m$ the number of even indexed entries and $t$ the number of odd indexed entries, by beginning with the arcless graph on $n$ vertices and linking first, the arcs $\{(v_i, v_j):(i+a_1) \ge j, i<j\}$. Clearly we have that $d_r^+(v_i) = 1$, for $i= 1, 2, 3, \ldots, n-1$. By recursively adding the linkage of arcs $\{(v_i, v_j):(i+a_s) \ge j, i<j\}$ for all, say $t$ odd indexed entries, we have that $d_r^+(v_i) = t$ for $i= 1, 2, 3, \ldots, n-1$. 

By applying Definition \ref{Def-OG} for all, say $m$ even index entries,the result $d_r^-(v_i) = t$ for $i= 1, 2, 3, \ldots, n-1$ follows in a similar manner.
\end{proof}
\begin{theorem}\label{T-OGRT1}[\textit{Ratanang's Theorem}]  For a Kyle Graph with $m \times m$ adjacency matrix
\begin{equation*} 
\begin{pmatrix}
0 & e_{12} & \dots & e_{1m}\\
e_{21} & 0 & \dots & e_{2m}\\
\hdotsfor{4}\\
e_{m1} & e_{m2} & \ \dots & 0
\end{pmatrix},
\end{equation*}
with each entry $e_{ij}$ equal to the number of arcs $(v_i, v_j)$, the ordered string defining the ornated graph $O_n(s_l), n\ge m$,  is given by $s_l = (a_j); 1 \le j \le l$ such that $a_1 < a_2 < a_3 < \ldots < a_l, l= (e_{i(i+1)} + e_{(i+1)i}), i\le (m-1)$ and $a_1, a_2, a_3, \ldots, a_l = \lfloor \frac {m}{2} \rfloor$ can all be determined.
\end{theorem}
\begin{proof}
The result $l= (e_{i(i+1)} + e_{(i+1)i}),i \le (m-1)$ follows directly from Proposition \ref{P-OGIOD}.

Consider the adjacency matrix of the Kyle Graph,
\begin{equation*} 
\begin{pmatrix}
0 & e_{12} & \dots & e_{1m}\\
e_{21} & 0 & \dots & e_{2m}\\
\hdotsfor{4}\\
e_{m1} & e_{m2} & \ \dots & 0
\end{pmatrix}.
\end{equation*}
Without loss of generality we assume that the defining ordered string $s_l = (a_j); 1 \le j \le l$ has odd number of entries. Consider row $(e_{1j}), 1\le j \le m$ (row 1) and count the number of entries $e_{1j} \ge 1, 1<j \le m$, say $q$. From Definition \ref{Def-OG} it follows that $a_l = q = \lfloor \frac {m}{2} \rfloor$. 

Now, consider the deviated matrix where we subtract $1$ from each first row entry, $e_{1j} \ge 1, (1<j \le m)$.  Consider first row of the deviated matrix and count the number of entries $e_{1j}^* \ge 1, 1<j \le m$, say $t$. From Definition \ref{Def-OG} it follows that $a_{(l-2)} = t$. Recursively all entries $a_1, a_3, a_5, \ldots, a_l$ can be determined.

Now, consider the column $(e_{j1}), 1\le j \le m$, (column 1) and count the number of entries $e_{j1} \ge 1, 1<j \le m$, say $t$. From Definition \ref{Def-OG} it follows that $a_{l-1} = t$. Now, consider the deviated matrix where we subtract $1$ from each first column entry, $e_{j1} \ge 1, 1< j \le m$. Consider the first column of the deviated matrix and count the number of entries $e_{j1}^* \ge 1, 1< j \le m$, say $w$.  From Definition \ref{Def-OG} it follows that $a_{(l-3)} = w$. Recursively all entries $a_2, a_4, a_6, \ldots, a_{l-1}$ can be determined. 

\ni This completes the proof.
\end{proof}

A relation between the degrees of vertices in an ornated graph and its adjacency matrix is established in the following result.

\begin{corollary}
For any vertex $v_i$ in an ornated graph, $d(v_i) = \sum \limits_{j=1}^{n}e_{ij} + \sum \limits_{j=1}^{n}e_{ji}$, where $e_{ij}, e_{ji}$ are the corresponding entries in the $n \times n$ adjacency matrix.
\end{corollary}
\begin{proof}
Because $d^+(v_i) = \sum \limits_{j=1}^{n}e_{ij}$ and $d^-(v_i) = \sum \limits_{j=1}^{n}e_{ji}$ the result $d(v_i) = d^+(v_i) + d^-(v_i) = \sum \limits_{j=1}^{n}e_{ij} + \sum \limits_{j=1}^{n}e_{ji}$, follows easily.
\end{proof}

\section{Generalised Ornated and Kyle Graphs}

We now introduce the concept of generalised ornated and Kyle graphs by applying a finite number of finite strings consecutively to $n \in \N$ vertices.

\begin{definition}{\rm 
When applying the finite number of ordered strings $s_{l_1}, s_{l_2}, \ldots s_{l_t}$ consecutively to the same $n$ vertices, we write $O_n(s_{l_1}, s_{l_2}, \ldots, s_{l_t}) = O_n(s_{l_1}) + O_n(s_{l_2}) + \ldots + O_n(s_{l_t}) = \sum\limits_{i=1}^{t}O_n(s_{l_i})$ and the graph is called the \textit{generalised ornated graph}.}
\end{definition}

\ni Let $k_i = max\{(a_j); a_j \in s_{l_i}\}$. In view of this definition, we establish the following theorem. 

\begin{theorem}\label{Thm-4.1}
For the finite number of ordered strings $s_{l_1}, s_{l_2}, \ldots, s_{l_t}$, the smallest ornated graph having the maximum degree $\Delta(O_n(s_{l_1}, s_{l_2},\ldots, s_{l_t})) = \sum\limits_{j=1}^{t}\Delta(O_n(s_{l_j}))$ has $n=2(\max\{ k_i\}_{\forall i}) + 1$ vertices. 
\end{theorem}
\begin{proof}
For the ordered string $s_{l_m}$ having $k_m = \max\{ k_i\}_{\forall i}$, we have the Kyle vertex $v_{k_m+1}$ with $\Delta(O_{2k_m+1}(s_{l_m}))=d_{O_{2k_m+1}}(v_{k_m+1})$. Without loss of generality, if we apply ordered string $s_{l_h}$ with $k_h \leq k_m$, then vertex $v_{k_m+1} \in \K(O_{2k_m+1}(s_{l_h}))$. Hence, $\Delta(O_{2k_m+1}(s_{l_m}, s_{l_h})) = d_{O_{2k_m+1}}(v_{k_m+1}) + 2k_h+1 = d_{O_{2k_m+1}}(v_{k_m+1}) + d_{O_{2k_h+1}}(v_{k_h+1}) = \Delta(O_{2k_m+1}(v_{k_m+1})) + \Delta(O_{2k_h+1}(v_{k_h+1}))= \Delta(O_{2k_m+1}(v_{k_m+1})) + \Delta(O_{2k_m+1}(v_{k_h+1}))$. Recursively the result follows.
\end{proof}

\begin{corollary}
If $n=1+2 \max\{ k_i\}$, then for any vertex $v_i \in V(O_n(s_{l_1}, s_{l_2}, \ldots, s_{l_t}))$ we have $d(v_i)=\sum\limits_{j=1}^{t}d_{O_n(s_{l_j})}(v_i)$.
\end{corollary}
\begin{proof}
The result follows immediately from the proof of Theorem \ref{Thm-4.1}.
\end{proof}

The Kyle graph corresponding to $O_n(s_{l_1}, s_{l_2}, \ldots, s_{l_t}), n = 2(max\{ k_i\}_{\forall i}) + 1$ is called the \emph{generalised Kyle graph}

\section{Conclusion and Scope for Further Studies}

In this paper, we have introduced a new family of directed graphs called ornated graphs and initiated a study on the structural properties and characteristics of ornated graphs. Some problems in this area are still to be settled and they seem to be much promising and challenging for further intensive investigation. Some of those problems we have identified for future studies in this are are the following. 

\ni Let us now introduce the notion of a symmetric directed graph as follows.

\begin{definition}{\rm
If the vertices of a directed graph $D$ on $n$ vertices can be numbered such that $d^+(v_1) \le d^+(v_2) \le \ldots \le d^+(v_i) \le d^+(v_{i+1}) = \ldots =d^+(v_{i+j})_{j-1 \ge 0}, \ge d^+(v_{n-(i-1)}) \ge d^+(v_{n-(i-2)}) \ge \ldots \ge d^+(v_n)$ and $d^-(v_1) \le d^-(v_2) \le \ldots \le d^-(v_i) \le d^-(v_{i+1}) = \ldots =d^-(v_{i+j})_{j-1 \ge 0} \ge d^-(v_{n-(i-1)}) \ge d^-(v_{n-(i-2)}) \ge \ldots \ge d^-(v_n)$, then $D$ is called a \textit{symmetric directed graph}.}
\end{definition}

\ni We strongly believe that the following conjecture on symmetric directed graphs is true.

\begin{conjecture}\label{Con-OGAM}
For a symmetric directed graph consider the corresponding adjacency matrix,
\begin{equation*} 
\begin{pmatrix}
e_{11} & e_{12} & \dots & e_{1m}\\
e_{21} & e_{21} & \dots & e_{2m}\\
\hdotsfor{4}\\
e_{m1} & e_{m2} & \ \dots & e_{mm}
\end{pmatrix}.
\end{equation*}
A symmetric directed graph is a Kyle Graph if and only if
\begin{enumerate}\itemsep0mm
\item[(i)] $m$ is odd and $m = 2a_l + 1, a_l$ as determined by Ratanang's Theorem,
\item[(ii)] $e_{ii} =0,\forall i$,
\item[(iii)] $e_{i(i+1)}+ e_{(i+1)i} = e_{j(j+1)} + e_{(j+1)j}, 1\le i,j \le (m-1)$,
\item[(iv)] For row $(e_{1j})$ we have $e_{12} \ge e_{13}\ge \ldots \ge e_{1\lfloor\frac {m}{2}\rfloor} =1$ and $e_{1j}=0, \lfloor\frac {m}{2}\rfloor +1 \le j \le m$,
\item[(v)]  For column $(e_{j1})$ we have $e_{21} = e_{31}= \ldots= e_{(\lfloor\frac {m}{2}\rfloor -1)1} = 1$ and $e_{j1} =0, \lfloor\frac {m}{2}\rfloor \le j \le m$.
\end{enumerate} 
\end{conjecture}

It can be seen that the first part of the conjecture is an immediate consequence of Lemma \ref{Lem-KG2} on Kyle graphs. Since no loops are included in the definition of ornated graphs, we can also note that $e_{ii} =0,\forall i$. The remaining parts of the above result remain to be settled.

\ni Some other open problems we have identified in this area are the following.

\begin{problem}{\rm  
Conclude the proof of Conjecture \ref{Con-OGAM}.}
\end{problem}

\begin{problem}{\rm  
Find the values of $d^+(v_i)$ and $d^-(v_i)$ for the ornated graph, $O_n((a_j); 1 \le j \le l))$ for $j \in \{1, 2, 3, \ldots, n\}$ in general.}
\end{problem} 

\begin{problem}{\rm  
Generalise Ratanang's Theorem allowing ordered strings such that $a_1 \le a_2 \le a_3 \le \ldots \le a_l$. }
\end{problem}

\begin{problem}{\rm 
Generalise Ratanang's Theorem with regard to the generalised ornated graph, $O_n(s_{l_1}, s_{l_2}, \ldots, s_{l_t})$, where $n=1+2\max\{ k_i\}$, where $1\le i\le l$ .}
\end{problem}

Now repetition and more entry values are possible. Hence, for example, $(1, 3, 5, 6)$ can become $(1, 1, 2 ,3, 4, 4, 4, 4, \ldots, 4, 5, 5, 6, 6, 6, \ldots, a_l=6), 1\le a_i\le 6, 1\le i\le l$.

\begin{problem}{\rm  
Given the degree sequence $d(v_1), d(v_2), d(v_3), \ldots, d(v_m)$ of a Kyle Graph, is it possible to determine an ordered string $s_l$ so that $O_n(s_l)$ obtains the exact degree sequence? At least it is obvious that $\nu(O_n(s_l)) = m$ and that, $a_l = \lfloor \frac {m}{2} \rfloor$.} 
\end{problem}

All these facts highlight that there is a wide scope in this area for further research.

\end{document}